\documentclass{amsart}

\RequirePackage{fix-cm}
\usepackage{graphicx}
\usepackage{amsmath, amsopn,amstext,amscd,amsfonts,amssymb}
\usepackage{dsfont}
\usepackage{comment}
\usepackage[active]{srcltx}
\usepackage{graphicx, epsfig, subfig}
\newcommand{\dps}{\displaystyle}

\usepackage{textcomp}

\usepackage{algorithm}

\makeatletter
\def\BState{\State\hskip-\ALG@thistlm}
\makeatother

\def\downbar#1{
\setbox10=\hbox{$#1$}
            \dimen10=\ht10 \advance\dimen10 by 2.5pt
            \ifdim \dimen10<15pt 
               \advance\dimen10 by -0.5pt
               \dimen11=\dimen10
               \advance\dimen10 by 2.5pt
               \lower \dimen11
            \else \lower \ht10 \fi
            \hbox {\hskip 1.5pt \vrule height \dimen10 depth \dp10}}
\def\upbar#1{
\setbox10=\hbox{$#1$}
            \dimen10=\ht10 \advance\dimen10 by \dp10 \advance\dimen10 by 2.5pt
            \ifdim \dimen10<15pt 
                \advance\dimen10 by 2pt \fi
            \raise 2.5pt \hbox {\hskip -1.5pt \vrule height \dimen10}}

\usepackage{multicol}
\usepackage{colortbl}
\usepackage{exscale}
\usepackage{amssymb,latexsym,amsthm,amsfonts,color,fancyhdr,mathrsfs}
\usepackage{lscape}
\usepackage{rotating}
\usepackage{caption}

\newtheorem{proposition}{\bf Proposition}[section]

\newtheorem{conjecture}{\bf Conjecture}[section]

\numberwithin{equation}{section}
\begin{document}

\title[A proof of a conjecture]{A proof of a conjecture  about a symmetric system of orthogonal polynomials}
\author{K. Castillo}
\address{CMUC, Department of Mathematics, University of Coimbra, 3001-501 Coimbra, Portugal}
\email{kenier@mat.uc.pt}
\author{M. N. de Jesus}
\address{CI$\&$DETS/IPV, Polytechnic Institute of Viseu, ESTGV, Campus Polit\'ecnico de Repeses, 3504-510 Viseu, Portugal}
\email{mnasce@estv.ipv.pt}
\author{J. Petronilho}
\address{CMUC, Department of Mathematics, University of Coimbra, 3001-501 Coimbra, Portugal}
\email{josep@mat.uc.pt}

\date{\today}

\begin{abstract}
The purpose of this note is to give an affirmative answer to a conjecture appearing in
[Integral Transforms Spec. Funct. 26 (2015) 90--95].
\end{abstract}
\maketitle

\section{Introduction}\label{1}
The following conjecture is one of the open problems collected by C. Berg during the international symposium on ``Orthogonal Polynomials, Special Functions and Applications" (OPSFA-12), Sousse, Tunisia, March 29, 2013 (see \cite[p. 90]{B15}):
\vspace{2mm}
\begin{changemargin}{0.8cm}{0.8cm}
{\em ``We consider the weight $w(x)=|x+\frac{1}{2}|/\sqrt{1+x}+|x-\frac{1}{2}|/\sqrt{1-x}$ on $[-1,1]$, and claim that there exits a sequence of polynomials $\{P_n\}_{n\geq0}$ orthogonal with respect to $w(x)$ on $[-1,1]$ and which fulfils
\begin{equation}\label{poly}
\begin{array}{l}
\quad P_0(x)=1, \quad P_1(x)=x\\[7pt]
P_{n+2}(x)=xP_{n+1}(x)-\gamma_{n+1}P_n(x), \quad n\geq0
\end{array}
\end{equation}
such that
\begin{align*}
\gamma_1=\frac{1}{2}, \quad \gamma_2=\frac{1}{4}, \quad \gamma_3=\frac{7}{30}, \quad \gamma_4=\frac{4}{15},\\[7pt]
\gamma_3+\gamma_4=\frac{1}{2},\\[7pt]
\gamma_5=\frac{1}{4}, \quad \gamma_6=\frac{12}{49}, \quad \gamma_7=\frac{25}{98},\\[7pt]
\gamma_6+\gamma_7=\frac{1}{2},\\[7pt]
\gamma_8=\frac{1}{4}, \quad \gamma_9=\frac{3187}{12870}, \quad \gamma_{10}=\frac{1624}{6435},\\[7pt]
\gamma_9+\gamma_{10}=\frac{1}{2},
\end{align*}
that is to say
\begin{align}\label{cond}
\gamma_1=\frac{1}{2}, \quad \gamma_{3n+2}=\frac{1}{4}, \quad \gamma_{3n+3}+\gamma_{3n+4}=\frac{1}{2}."
\end{align}
}
\end{changemargin}
\vspace{1.5mm}
According to information provided in \cite{B15}, the problem was raised by M. J. Atia (see also \cite[p. 46]{AL12}), not only in OPSFA-12 but also in OPSFA-10 (Leuven, 2009) and OPSFA-11 (Madrid, 2011). However, the conjecture is ill-posed because \eqref{cond} does not necessarily implies
\begin{align}\label{coef}
\gamma_3=\frac{7}{30}, \quad \gamma_4=\frac{4}{15}, \quad \gamma_6=\frac{12}{49},\quad \gamma_7=\frac{25}{98}, \quad \gamma_9=\frac{3187}{12870}, \quad \gamma_{10}=\frac{1624}{6435}\;.
\end{align}
Proposition \ref{prop} in below shows that if \eqref{cond} holds, then the associated weight function is not necessarily given by
\begin{align}\label{weight}
\dps\frac{\left|x+\frac12\right|}{\dps\sqrt{1+x}}+\frac{\left|x-\frac12\right|}{\sqrt{1-x}}\;,\quad -1< x<1\;.
\end{align}
We thus reformulate the conjecture as follows:
\begin{conjecture}\label{conjecture}
There is a sequence of positive numbers $(\gamma_n)_{n\geq 1}$ fulfilling \eqref{cond} and \eqref{coef} such that the sequence of polynomials $(P_n)_{n\geq 0}$ given by \eqref{poly} is orthogonal with respect to \eqref{weight}.
\end{conjecture}

In Section \ref{2}, we link Conjecture \ref{conjecture} with polynomial mappings and, in Section \ref{3}, arrive at an affirmative answer.

\section{A symmetric system of orthogonal polynomials}\label{2}
Let $\widehat{T}_{3}$ and $\widehat{U}_{2}$ denote the monic Chebyshev polynomials of the
first and second kind, respectively, given by
$$
\widehat{T}_3(x):=x^3-\frac34 x=\frac14\,\cos(3\theta)\;,\quad
\widehat{U}_2(x):=x^2-\frac14=\frac14\frac{\sin (3\theta)}{\sin\theta},
\quad \theta=\arccos x\,.
$$

\begin{proposition}\label{prop}
Define $\gamma_0:=0$ and let $(\gamma_n)_{n\geq1}$ be a sequence of nonzero complex numbers such that
\begin{align*}
\gamma_{3n}+\gamma_{3n+1}=\frac12\;, \quad
\gamma_{3n+2}=\frac14\;,
\end{align*}
for each nonnegative integer $n$. Let $(P_n)_{n\geq0}$ be the sequence of monic orthogonal polynomials given by
$$
P_{n+1}(x)=xP_n(x)-\gamma_nP_{n-1}(x)\;.
$$
Then
\begin{align*}
P_{3n}(x)&=Q_{n}\big(\widehat{T}_{3}(x)\big)\;,\\[7pt]
P_{3n+1}(x)&=\displaystyle \frac{Q_{n+1}\big(\widehat{T}_{3}(x)\big)+\gamma_{3n+1}xQ_{n}\big(\widehat{T}_{3}(x)\big)}{\widehat{U}_2(x)}\;, \\[7pt]
P_{3n+2}(x)&=\displaystyle
\frac{xQ_{n+1}\big(\widehat{T}_{3}(x)\big)+\frac14\gamma_{3n+1}Q_{n}\big(\widehat{T}_{3}(x)\big)}{\widehat{U}_2(x)}\;,
\end{align*}
where $(Q_n)_{n\geq0}$ is the sequence of monic orthogonal polynomials given by
\begin{align*}
Q_{n+1}(x)=xQ_n(x)-\frac14\,\gamma_{3n-2}\gamma_{3n}\,Q_{n-1}(x)\;.
\end{align*}
Assume furthermore that $\gamma_n>0$ for each positive integer $n$.
Then $(P_n)_{n\geq0}$ and $(Q_n)_{n\geq0}$ are orthogonal polynomial sequences with respect to
certain positive measures, say $\mu_P$ and $\mu_Q$ respectively.
Suppose that $\mu_Q$ is absolutely continuous
with weight function $w_Q$ on $[\xi,\eta]$ with $-1/4\leq \xi<\eta\leq 1/4$.
Then $\mu_P$ is also absolutely continuous with weight function
$$
w_P(x):=\big|\widehat{U}_{2}(x)\big|w_Q\big(\widehat{T}_{3}(x)\big)
$$
on $E:=\widehat{T}_{3}^{-1}\big([\xi,\eta]\big)$.
\end{proposition}

\begin{proof}
The proof is a straightforward application of \cite[Theorem 2.1, Theorem 3.4, and Remark 3.5]{MP10},
choosing therein $(k,m,r)=(3,0,0)$ and
$$
a_n^{(j)}:=\gamma_{3n+j}\;,\quad b_n^{(j)}:=0\;,
$$
for each nonnegative integer $n$ and $j\in\{0,1,2\}$.
Under these conditions, the polynomials $\pi_k$, $\theta_m$, and $\eta_{k-1-m}$
defined in \cite[Theorem 2.1]{MP10} are given by
$$
\pi_3=\widehat{T}_3\;,\quad \theta_0=1\;,\quad \eta_2=\widehat{U}_2\;,
$$
which completes the proof.
\end{proof}

\section{Proof of Conjecture \ref{conjecture}}\label{3}


Let $w: (-1,1)\to [4, \infty)$ be given by
\begin{align*}
&w(x):=\\
&\frac{1}{\left(\displaystyle\cos\frac{\arccos x}{3}-\frac12\right)\sqrt{\displaystyle 1+\cos\frac{\arccos x}{3}}}+\frac{1}{\left(\displaystyle\cos\frac{\arccos x}{3}+\frac12\right)\sqrt{\displaystyle 1-\cos\frac{\arccos x}{3}}}\;.
\end{align*}
Let $w_Q$ be the weight function (supported) on $\big[-1/4,1/4\big]$ defined by
$$
w_Q(x):=w(4 x)\;,\quad -\frac14< x<\frac14\,.
$$
We see at once that 
\begin{align*}
&w(\cos\theta)\\
&=\frac{1}{\left(\cos\dps\frac{\theta}{3}-\dps\frac12\right)\sqrt{1+\cos\dps\frac{\theta}{3}}}
+\frac{1}{\left(\cos\dps\frac{\theta}{3}+\dps\frac12\right)\sqrt{1-\cos\dps\frac{\theta}{3}}}\;, \\[7pt]
& =\frac{1}{2\sqrt{2}}\left(\sec\dps\frac{\theta-2\pi}{6}\sec\dps\frac{\theta+2\pi}{6}\sec\dps\frac{\theta}{6}
+\sec\dps\frac{\theta-\pi}{6}\sec\dps\frac{\theta+\pi}{6}\csc\dps\frac{\theta}{6}\right)\;, \quad\theta\in\big(0,\pi\big)\;.
\end{align*}
This allows us to show that $w$ is an even function on $(-1,1)$. 
Moreover,
\begin{align*}
&w\big(\cos(3\theta)\big)=\\[7pt]
&\frac{1}{\left|\cos\theta-\dps\frac12\right|\sqrt{1+\cos\theta}}+\frac{1}{\left|\cos\theta+\dps\frac12\right|\sqrt{1-\cos\theta}}\;,\quad \theta\in\big(0,\pi\big)\setminus\left\{\frac{\pi}3,\frac{2\pi}3\right\}\,.
\end{align*}
The careful reader should carry through the details of this little calculation, analyzing the cases
$0<3\theta<\pi$, $\pi<3\theta<2\pi$, and $2\pi<3\theta<3\pi$ separately.
Writing $x=\cos \theta$, $0<\theta<\pi$, we obtain
$$
\big|\widehat{U}_{2}(x)\big|w_Q\big(\widehat{T}_{3}(x)\big)
=\frac{\left|x+\frac12\right|}{\sqrt{1+x}}+\frac{\left|x-\frac12\right|}{\sqrt{1-x}}\;,\quad -1<x<1\,.
$$
Define the linear functional $\mathcal{L}: \mathbb{R}[x]\to\mathbb{R}$ by
$$
\mathcal{L}[f]:=\int_{-1/4}^{1/4} f(x) w_Q(x)\,\mathrm{d}x=\int_{0}^{\pi} f\Big(\frac{\cos\theta}{4}\Big) w(\cos\theta)\sin\theta\,{\rm d}\theta\,.
$$
Set $\Delta_{-1}:=1$ and $\Delta_n:=\det(\mu_{i+j})_{i,j=0}^n$, where $\mu_n:=\mathcal{L}[x^n]$ for each nonnegative integer $n$. Obviously, $\mathcal{L}$ is positive definite (see \cite[Definition $3.1$, p. $13$]{C78}) and, therefore, $\Delta_n>0$ (see \cite[Theorem $3.4$, p. $15$]{C78}). Define
$$
s_0:=0, \quad s_n:=\frac{\Delta_n \Delta_{n-2}}{\Delta_{n-1}^2}\;,
$$
for each positive integer $n$. Since the weight function $w_Q$ is even, $\mathcal{L}$ is symmetric (see \cite[Definition $4.1$, p. $20$]{C78}). By \cite[Theorem $4.2$, p. $19$ and Theorem $4.3$, p. $21$]{C78}, it follows that the sequence $(Q_n)_{n\geq 0}$ of monic orthogonal polynomials with respect to $\mathcal{L}$ satisfies
\begin{align*}
Q_{n+1}(x)=xQ_n(x)-s_n Q_{n-1}(x)\;.
\end{align*}
Since the true interval of orthogonality (see \cite[Definition $5.2$, p. $29$]{C78}) of $\mathcal{L}$ is $[\xi_1,\eta_1]=[-1/4,1/4]$, $(16\, s_n)_{n\geq 1}$ is a chain sequence by \cite[Theorem $2.1$, p. $108$]{C78}. Thus there exists a sequence $(g_n)_{n\geq 0}$ such that (see \cite[Definition $5.1$, p. $91$]{C78})
\begin{align*}
0\leq g_0<1\;,\quad 0<g_n<1\;,\quad 16\,s_n=(1-g_{n-1})g_n\;,
\end{align*}
for each positive integer $n$. By \cite[Theorem $5.2$, p. $93$]{C78} we may assume $g_0=0$.
Define a sequence $(\gamma_n)_{n\geq 0}$ as follows:
\begin{align}\label{gammas-good}
\gamma_{3n}:=\frac{1}{2}\;g_n, \quad \gamma_{3n+1}:=\frac{1}{2}\,(1-g_n), \quad \gamma_{3n+2}:=\frac14\;.
\end{align}
Clearly, this sequence satisfies \eqref{cond} and $s_n=1/4\,\gamma_{3n-2}\gamma_{3n}$
for each positive integer $n$. Moreover, it also satisfies \eqref{coef}.
Indeed, since $\mathcal{L}$ is a symmetric functional, $\mu_{2n+1}=0$ for each nonnegative integer $n$.
Furthermore,
$$
\mu_0=\mathcal{L}[1]=\int_0^\pi w(\cos\theta)\sin\theta\,{\rm d}\theta=2\sqrt{2}
$$
and, similarly,
$$
\mu_2=\frac{7\sqrt{2}}{120}\;,\quad
\mu_4=\frac{107\sqrt{2}}{40320}\;,\quad
\mu_6=\frac{835\sqrt{2}}{6150144}\;.
$$
Therefore,
$$
\Delta_0=2\sqrt{2}\;,\quad
\Delta_1=\frac{7}{30}\;,\quad
\Delta_2=\frac{\sqrt{2}}{4500}\;,\quad
\Delta_3=\frac{3187}{476756280000}\;,
$$
and so 
$$
s_1=\frac{7}{240}\;,\quad
s_2=\frac{4}{245}\;,\quad
s_3=\frac{15935}{1009008}\;.
$$
Consequently, since $g_n=16s_n/(1-g_{n-1})$ for each positive integer $n$, we get
$$
g_0=0\;,\quad
g_1=\frac{7}{15}\;,\quad
g_2=\frac{24}{49}\;,\quad
g_3=\frac{3187}{6435}\;.
$$
This allows us to verify that the sequence $(\gamma_n)_{n\geq 0}$ defined by \eqref{gammas-good} indeed fulfils \eqref{coef}.
Finally, Proposition \ref{prop} implies the truthfulness of Conjecture \ref{conjecture}.

\section*{Acknowledgements}
KC and JP are supported by the Centre for Mathematics of the University of Coimbra --UID/MAT/00324/2019, funded by the Portuguese Government through FCT/MEC and co-funded by the European Regional Development Fund through the Partnership Agreement PT2020.
MNJ supported by UID/Multi/04016/2019, funded by FCT. MNJ also thanks the Instituto Polit\'ecnico de Viseu and CI\&DETS for their support.

\end{document}